\documentclass[preprint,12pt]{article}
\usepackage{amsmath,amsthm,amssymb,amscd,fancybox,epsfig,float}
\usepackage[margin=1in]{geometry}
\usepackage{authblk}

\usepackage{hyperref}
\usepackage{graphicx}
\usepackage{bbm}
\usepackage{amsmath,amsthm,amssymb,amscd}
\usepackage{caption,subcaption}
\usepackage{booktabs,multirow}
\usepackage{setspace} 
\usepackage{color}
\usepackage{array}

\usepackage{arydshln}
\usepackage{graphics}
\usepackage{epstopdf}
\usepackage{multirow}
\usepackage{hhline}
\usepackage{cleveref}
\usepackage{xcolor}
\usepackage{tikz}
\usetikzlibrary{arrows,automata}
\usepackage{cite}
\newcommand{\beqn}{\begin{equation}}
\newcommand{\eeqn}{\end{equation}}

\newcommand{\bbmat}{\begin{bmatrix}}
\newcommand{\ebmat}{\end{bmatrix}}

\usepackage{bm}
\usepackage{enumerate}

\usepackage{array}
\usepackage{makecell}
\usepackage{enumerate}

\usepackage{dashbox}

\newcounter{algo}

\newtheorem{algorithm}[algo]{Algorithm}
\newtheorem{definition}{Definition}[section]
\newtheorem{remark}{Remark}[section]
\newtheorem{remark1}{Remark}[section]

\newtheorem{lemma}{Lemma}[section]

\title{A fast adaptive algorithm for scattering from a two dimensional radially-symmetric potential}
\author{Jeremy Hoskins\thanks{
Department of Statistics, University of Chicago, Chicago, IL 60637.
email: jeremyhoskins@uchicago.edu} ,\, 
Vladimir Rokhlin\thanks{Department of Mathematics and Department of Computer Science, Yale University, New Haven, CT, 06511. 
email: rokhlin@cs.yale.edu}}

\date{}

\begin{document}

\maketitle
\abstract{
 In the present paper we describe a simple black box algorithm for efficiently and accurately solving scattering problems related to the scattering of time-harmonic waves from radially-symmetric potentials in two dimensions. The method uses FFTs to convert the problem into a set of decoupled second-kind Fredholm integral equations for the Fourier coefficients of the scattered field. Each of these integral equations are solved using scattering matrices, which exploit certain low-rank properties of the integral operators associated with the integral equations. The performance of the algorithm is illustrated with several numerical examples including scattering from singular and discontinuous potentials. Finally, the above approach can be easily extended to time-dependent problems. After outlining the necessary modifications we show numerical experiments illustrating the performance of the algorithm in this setting.
}

%
%

\section{Introduction}  \label{30}
The scattering of waves from potentials is ubiquitous in applied mathematics and physics, arising {\it inter alia} in geophysics, medical imaging, non-destructive industrial testing, optics, etc. Typically in such applications it is highly desirable to be able to simulate the scattering from a given medium quickly and accurately. This is especially important for solving associated inverse problems; namely, to recover properties of the material from measurements of the scattered field outside the object. Inversion algorithms often require solving the forward problem (ie. determining the scattered field for a given medium) hundreds or thousands of times.

In this paper we describe a fast, adapative, simple, and accurate method for computing the scattering from a radially-symmetric body in two dimensions. The applications are two-fold. Firstly, radially-symmetric geometries are frequently encountered in applications. Secondly, the speed and accuracy of the proposed method allow one easily to validate new algorithms for solving the two-dimensional forward and inverse problems for the Helmholtz equation (or wave equation). See \cite{Gillman2015,les2016,vico2016} and the references therein for a discussion of fast algorithms for general two-dimensional scattering problems.

The remainder of this paper is organized as follows. In Section \ref{prelims} we describe the model and introduce the necessary mathematical tools. In Section \ref{sec:algo} we describe the algorithm for two-dimensional scatttering from radially-symmetric potentials. Numerical illustrations of the algorithm both for fixed frequencies and in the time domain are given in Section \ref{sec:nume}. Finally, Section \ref{sec:concl} discusses future work.

\section{Mathematical Preliminaries}\label{prelims} 
In the frequency domain the displacement in an inhomogeneous fluid satisfies the Helmholtz equation
\begin{align}\label{eqn:2dhelm}
\Delta u({\bf r}) + k^2 (1+Q({\bf r})) u({\bf r}) = 0.
\end{align}
In this paper we restrict our attention to two-dimensional problems for which the potential $Q$ is radially symmetric and compactly supported. The classical approach is to decompose the total field $u$ into the sum of an incident field $u_i$ and a scattered field $u_s.$ In typical problems the incident field $u_i$ is known and satisfies the Helmholtz equation with $Q \equiv 0.$ The scattered field $u_s$ satisfies
\begin{align}
\Delta u_s({\bf r}) + k^2 (1+q({\bf r})) u_s({\bf r}) = -k^2 q({\bf r})u_i({\bf r}),
\end{align}
together with the Sommerfeld radiation condition
\begin{align}
\lim_{r \to \infty} \sqrt{r}\left[\frac{\partial}{\partial r}u_s({\bf r})-ik u_s({\bf r})\right] = 0.
\end{align}

We observe that since the potential $Q$ is radially-symmetric and compactly supported there exists a compactly supported function $q:[0,\infty)$ such that $Q({\bf r}) = q(\|{\bf r}\|)$ for all ${\bf r} \in \mathbb{R}^2,$ where $\|\cdot\|$ denotes the standard Euclidean norm. With some abuse of notation we will also refer to the function $q$ as the potential.

\subsection{Reduction to the radial problem}
In this section we describe the reduction of the radially-symmetric acoustic scattering problem to a set of decoupled one-dimensional scattering problems. To that end, we let $u_m: [0,\infty) \rightarrow \mathbb{C}$ denote the Fourier coefficient of the scattered field $u_s$ with respect to the angle $\theta,$ i.e.
\begin{align}\label{eqn:umdef1}
u_m(r) = \int_0^{2\pi} e^{-im\theta}u_s(r \cos\theta,r \sin\theta)\,{\rm d}\theta.
\end{align}
It is easily shown that $u_m$ satisfies the following ordinary differential equation
\begin{align}\label{eqn:1d}
u_m''(r) + \frac{1}{r}u_m'(r)+k^2[1+q(r)]u_m(r)-\frac{m^2}{r^2}u_m(r)= - k^2 q(r) f_m(r),
\end{align}
where 
\begin{align}
f_m(r) = \int_0^{2\pi} e^{-im\theta}u_i(r \cos\theta,r \sin\theta)\,{\rm d}\theta,
\end{align}
is the Fourier coefficient of the incident field $u_i.$

The following remarks and lemmas summarize results pertaining to the solutions of (\ref{eqn:1d}) (see, for example, \cite{abram}).

\begin{remark1}
On any interval $c<x<d$ on which the source $f$ and the potential $q$ are identically zero the solutions to equation (\ref{eqn:1d}) are a linear combination of the Bessel function $J_m(kr)$ and the Hankel function $H_m(kr).$ 
\end{remark1}

The following lemma characterizes the solution to equation (\ref{eqn:1d}) outside the support of $q.$
\begin{lemma}
Suppose that the potential $q$ is supported on the interval $[a,b]$ and that $u_m$ is the function defined by (\ref{eqn:umdef1}). Then there exists some constant $\mu$ such that $u_m(r) = \mu H_m(kr)$ for all $r \ge b.$ Similarly, there exist constants $A,\alpha$ such that
$$u_m(r) =  \alpha J_m(kr),$$
for all $0 \le r \le a.$ 
\end{lemma}

The following lemma provides the Green's function for the homogeneous equation corresponding to (\ref{eqn:1d}), obtained by replacing the right-hand side by an arbitrary function $g \in L^2[a,b]$ and setting $q$ to zero on the left-hand side.
\begin{lemma}
Let $g$ be a function compactly supported in the interval $[a,b].$ For all $m\ge 0$ the function $u_m$ defined by the following formula
\begin{equation}
u_m(r) = -\frac{i\pi}{2}\left[\int_0^r J_m(kt) H_m(kr)  \,g(t)\, t \,{\rm d}t + \int_r^\infty J_m(kr) H_m(kt) \,g(t) \,t\,{\rm d}t\right]
\end{equation}
satisfies the ordinary differential equation
\begin{equation}
u_m''(r)+\frac{1}{r} u_m'(r)+ k^2 u_m(r)-\frac{m^2}{r^2}u_m(r)= g(r).
\end{equation}

\end{lemma}
We conclude this section with the following definition.
\begin{definition}
For any integer $m$ we define the corresponding function $K_m :(0,\infty)\times (0,\infty) \to \mathbb{C}$ by
\begin{align}\label{eqn:kerdef}
K_m(r,t) = \begin{cases}
 J_m(kt) H_m(kr)  \, t  & {\rm if} \, t \le r,\\
 J_m(kr) H_m(kt)  \, t  & {\rm if} \, t \ge r, 
\end{cases}
\end{align}
Except when necessary we will suppress the subscript $m.$
\end{definition}

\subsection{Integral equations for modes}
In this section we derive integral equations for the Fourier components of the scattered field $u_m,$ $m=0,\pm 1,\pm 2,\dots$  for variable $q.$ The two-point boundary value problem for $u_m$ can be converted into a second-kind integral equation for a new unknown $\rho_m$ by writing $u_m$ in the form
\begin{align}
u_m(r)= \int_0^\infty K(r,t) \rho_m(t)\,{\rm d}t,
\end{align}
Upon substitution of this formula into equation (\ref{eqn:1d}) we obtain
\begin{align}\label{eq:integral_eq}
\rho_m(r) + k^2 q(r) \int_0^\infty K(r,t) \rho_m(t)\,{\rm d}t= -k^2 q(r) f_m(r).
\end{align}
After $\rho_m$ is determined, $u_m$ can be obtained by integration against the kernel $K(r,t)$ defined in (\ref{eqn:kerdef}).

\begin{remark}
From (\ref{eq:integral_eq}) it is clear that the support of $\rho_m$ is contained in the support of $q.$
\end{remark}

\subsection{Scattering matrix formulation}
In this section we define scattering matrices, as well as incoming and outgoing expansion coefficients. Additionally, we provide expressions relating the expansion coefficients and scattering matrices on two adjacent disjoint intervals to the expansion coefficients and scattering matrices on their union. A similar formalism for more general two-point boundary value problems was developed in \cite{rokles}.

Given an interval $A=[a,b]$ and a function $\rho\in L^2[0,\infty)$ we define the left and right outgoing expansion coefficients $\alpha_\ell(A)$ and $\alpha_r(A),$ respectively, by
\begin{align}\label{eqn:alp_def1}
\alpha_\ell(A) &= \int_A J_m(kr) \rho_m(r) r {\rm d}r\\
\label{eqn:alp_def2}
\alpha_r(A) &= -i\frac{\pi}{2}\int_A H_m(kr) \rho_m(r) r {\rm d}r.
\end{align}
Similarly, we define the left and right incoming expansion coefficients $\phi_\ell(A)$ and $\phi_r(A),$ respectively, by
\begin{align}\label{eqn:phi_def1}
\phi_\ell(A) &= \int_0^a J_m(kr) \rho_m(r) r {\rm d}r,\\
\label{eqn:phi_def2}
\phi_r(A) &= -i\frac{\pi}{2}\int_b^\infty H_m(kr) \rho_m(r) r {\rm d}r.
\end{align}
In addition, we let $F_A: L^2[a,b]\to L^2[a,b]$ denote the operator defined by
\begin{equation}\label{eqn:def_F}
F_A(\rho)(r) = \rho(r) +k^2q(r)\int_a^b K(r,t) \rho(t)\,{\rm d}t.
\end{equation}
Next, we define the mapping $V_A: [a,b] \to \mathbb{R}$ by
\begin{align}\label{eqn:Vdef}
V_A(r) = -k^2 F_A^{-1}\left(q f\right)(r),
\end{align}
and the coefficients $\chi_J(A)$ and $\chi_H(A)$ by
\begin{align}\label{eqn:chi_def1}
\chi_J(A) &= \int_A J_m(kr) V_A(r) \,r \,{\rm d}r\\
\label{eqn:chi_def2}
\chi_H(A) &=-\frac{i\pi}{2} \int_A H_m(kr) V_A(r) \, r\, {\rm d}r.
\end{align}

Finally, we denote the restriction of $J_m(kr)$ and $-\frac{i\pi}{2}H_m(kr)$ to the interval $A$ by $J|_A$ and $H|_A,$ respectively, or simply by $J$ and $H,$ respectively, when there is no ambiguity.

The integral equation for $\rho_m$ on an interval $A=[a,b]$ can be written in terms of $\phi_\ell(A),$ $\phi_r(A),$ $V_A$ and $F_A,$ as shown in the following lemma. It also relates these quantities to the outgoing expansion coefficients of $A,$ $\alpha_{\ell,r}(A).$ Its proof is straightforward and omitted.
\begin{lemma}
Suppose that $m$ is a non-negative integer and that $\rho_m$ satisfies the integral equation (\ref{eq:integral_eq}). For an interval $A = [a,b]$ with $0 \le a <b <\infty$ suppose that $\alpha_\ell(A),$ $\alpha_r(A),$ $\phi_\ell(A),$ $\phi_r(A),$ $\chi_J(A)$ and $\chi_H(A)$ are the quantities defined in  (\ref{eqn:alp_def1}), (\ref{eqn:alp_def2}),(\ref{eqn:phi_def1}), (\ref{eqn:phi_def2}), (\ref{eqn:chi_def1}) and (\ref{eqn:chi_def2}), respectively. Further suppose that $F_A:L^2[a,b]\to L^2[a,b]$ is the operator defined by (\ref{eqn:def_F}) and $V_A$ is the function defined by (\ref{eqn:Vdef}). Then
\begin{align}\label{eqn:get_rho}
\rho_m(r) = -\phi_\ell(A) F_A^{-1}(H)(r)-\phi_r(A) F_I^{-1}(J)(r)+ V_A(r).
\end{align}
Moreover,
\begin{align}\label{eqn:alpha_eq2}
\begin{pmatrix}\alpha_\ell(A) \\ \alpha_r(A) \end{pmatrix} =  -\begin{pmatrix} \left<J, F_A^{-1}H \right>_A& \left< J, F_A^{-1}J \right>_A \\ \left< H, F_A^{-1}H \right>_A &
 \left< H, F_A^{-1}J \right>_A 
\end{pmatrix}
\begin{pmatrix}
\phi_\ell(A) \\ \phi_r(A)
\end{pmatrix}
+
\begin{pmatrix}
\chi_J(A)\\
\chi_H(A)
\end{pmatrix},
\end{align}
where $\left< \cdot, \cdot \right>_A$ denotes the standard $L^2$ inner product restricted to $A.$
\end{lemma}

\begin{definition}
The scattering matrix $S_A$ of the interval $A$ is the mapping from the incoming expansion coefficients to the outgoing expansion coefficients in the absence of a source. Specifically, using (\ref{eqn:alpha_eq2}) one finds that
\begin{align}
S_A =  -\begin{pmatrix} \left<J, F_A^{-1}H \right>_A& \left< J, F_A^{-1}J \right>_A \\ \left< H, F_A^{-1}H \right>_A &
 \left< H, F_A^{-1}J \right>_A 
\end{pmatrix}.
\end{align}
\end{definition}

The following lemma relates the expansion coefficients on the union of two intervals to the expansion coefficients on each of the subintervals.

\begin{lemma}\label{lem:merge_lem}
Suppose that $A = [a,b]$ and $B=[b,c]$ for some $0 \le a <b<c<\infty.$ Then
\begin{align}
\alpha_\ell(A\cup B) &= \alpha_\ell(A) + \alpha_\ell(B)\\
\alpha_r(A\cup B) &= \alpha_r(A) + \alpha_r(B)\\ 
\phi_\ell(A \cup B) &= \phi_\ell(A)\\
\phi_r(A \cup B) &= \phi_r(B).
\end{align}
Moreover, if
\begin{align}
L =
\begin{pmatrix} 1 & 0 \\0&1\\1&0\\0&1 \end{pmatrix},
\quad {\rm and} \quad K_{AB}=
\begin{pmatrix}
1&0 &0&(S_A)_{1,1}\\
0&1 &0&(S_A)_{2,1}\\
(S_B)_{2,1}&0 &1&0\\
(S_B)_{2,2}&0 &0&1
\end{pmatrix},
\end{align}
then
\begin{align}
S_{I \cup J} = 
L^T K_{AB}^{-1} \begin{pmatrix} S_A &0 \\ 0&S_B \end{pmatrix} L,\quad 
{\rm and} \quad
\begin{pmatrix} \chi_J (A\cup B) \\ \chi_H(A \cup B) \end{pmatrix} = 
L^TK_{AB}^{-1} \begin{pmatrix} \chi_J (A) \\ \chi_H(A) \\ \chi_J(B) \\ \chi_H(B)\end{pmatrix}.
\end{align}
Finally, if $\chi_J(A),$ $\chi_H(A),$$\chi_J(B),$ $\chi_H(B),$ $\phi_{\ell,r}(A\cup B),$ $S_A,$ $S_B$ and $S_{A\cup B}$ are known then $\alpha_{\ell,r}(A),$ $\phi_{\ell,r}(A),$ $\alpha_{\ell,r}(B),$ and $\phi_{\ell,r}(B)$ can be determined via the following formulas
\begin{align}
\begin{pmatrix}
\alpha_\ell (A) \\ \alpha_r(A) \\ \alpha_\ell(B) \\ \alpha_r(B) 
\end{pmatrix} =K_{AB}^{-1}
\begin{pmatrix} S_A &0 \\ 0&S_B \end{pmatrix}L
\begin{pmatrix}
\phi_\ell (A \cup B)\\ \phi_r(A \cup B) 
\end{pmatrix}+
\begin{pmatrix} \chi_J (A) \\ \chi_H(A) \\ \chi_J(B) \\ \chi_H(B)\end{pmatrix}
\end{align}
\begin{align}
\phi_\ell(A) &= \phi_\ell(A \cup B)\\
\phi_r (A) &= \phi_r (A\cup B) - \alpha_r (B)\\
\phi_\ell (B) &= \phi_\ell (A \cup B) - \alpha_\ell (A)\\
\phi_r(B) &= \phi_r(A \cup B).
\end{align}
\end{lemma}

\begin{proof}
From the definitions of $\phi_{\ell,r}$ it is clear that
\begin{align}
\phi_\ell(A) &= \phi_\ell(A \cup B) \label{eq:phiij1},\\
\phi_\ell (B) &= \phi_\ell (A \cup B) - \alpha_\ell (A)\label{eq:phiij2},\\
\phi_r (A) &= \phi_r (A\cup B) - \alpha_r (B)\label{eq:phiij3},\\
\phi_r(B) & = \phi_r(A \cup B) \label{eq:phiij4}
\end{align}

Moreover, applying equation (\ref{eqn:alpha_eq2}) to the intervals $A$ and $B$ yields
\begin{align}
\begin{pmatrix}
\alpha_\ell (A) \\ \alpha_r(A) \\ \alpha_\ell(B) \\ \alpha_r(B)
\end{pmatrix}=
\begin{pmatrix} S_A &0 \\ 0&S_B \end{pmatrix}
\begin{pmatrix} \phi_\ell (A) \\ \phi_r(A) \\ \phi_\ell(B) \\ \phi_r(B)\end{pmatrix}+
\begin{pmatrix} \chi_J (A) \\ \chi_H(A) \\ \chi_J(B) \\ \chi_H(B)\end{pmatrix}.
\end{align}
Using identities (\ref{eq:phiij1})-(\ref{eq:phiij4}) we obtain
\begin{flalign}\label{eqn:scat111}
\begin{pmatrix}
1&0 &0&(S_A)_{1,1}\\
0&1 &0&(S_A)_{2,1}\\
(S_B)_{2,1}&0 &1&0\\
(S_B)_{2,2}&0 &0&1
\end{pmatrix}\begin{pmatrix}
\alpha_\ell (A) \\ \alpha_r(A) \\ \alpha_\ell(B) \\ \alpha_r(B)
\end{pmatrix} = \begin{pmatrix} S_A &0 \\ 0&S_B \end{pmatrix}
\begin{pmatrix}
1 & 0\\
0 & 1\\
1 & 0\\
0& 1
\end{pmatrix}
\begin{pmatrix}
\phi_\ell (A \cup B)\\ \phi_r(A \cup B) 
\end{pmatrix}
+\begin{pmatrix} \chi_J (A) \\ \chi_H(A) \\ \chi_J(B) \\ \chi_H(B)\end{pmatrix}.&&
\end{flalign}
Defining
\begin{align}
L =
\begin{pmatrix} 1 & 0 \\0&1\\1&0\\0&1 \end{pmatrix},
\quad {\rm and} \quad K_{AB}=
\begin{pmatrix}
1&0 &0&(S_A)_{1,1}\\
0&1 &0&(S_A)_{2,1}\\
(S_B)_{2,1}&0 &1&0\\
(S_B)_{2,2}&0 &0&1
\end{pmatrix},
\end{align}
and substituting them into (\ref{eqn:scat111}) yields
\begin{align}
\begin{pmatrix}
\alpha_\ell (A \cup B)\\ \alpha_r(A \cup B) 
\end{pmatrix}
= L^T K_{AB}^{-1} \begin{pmatrix} S_A &0 \\ 0&S_B \end{pmatrix} L
\begin{pmatrix}
\phi_\ell (A \cup B)\\ \phi_r(A \cup B) 
\end{pmatrix}
+L^TK_{AB}^{-1} \begin{pmatrix} \chi_J (A) \\ \chi_H(A) \\ \chi_J(B) \\ \chi_H(B)\end{pmatrix}.
\end{align}
Comparing with (\ref{eqn:alpha_eq2}) it follows that
\begin{align}
S_{A \cup B} = 
L^T K_{AB}^{-1} \begin{pmatrix} S_A &0 \\ 0&S_B \end{pmatrix} L,\quad 
{\rm and} \quad
\begin{pmatrix} \chi_J (A\cup B) \\ \chi_H(A \cup B) \end{pmatrix} = 
L^TK_{AB}^{-1} \begin{pmatrix} \chi_J (A) \\ \chi_H(A) \\ \chi_J(B) \\ \chi_H(B)\end{pmatrix}.
\end{align}

\end{proof}

\section{The algorithm}\label{sec:algo}
In this section we describe an algorithm for solving the Helmholtz scattering problem from a radially-symmetric potential in two dimensions
\begin{align}
\Delta u +k^2 (1+q(\|{\bf r}\|) u = -k^2 q(\|{\bf r}\|) u_i({\bf r})
\end{align}
where $u_i$ is the incident field, $q$ is supported on the interval $[a,b],$ and the scattered field $u$ satisfies the Sommerfeld radiation condition
\begin{align}
\lim_{r\to \infty} \sqrt{r} \left[ \frac{\partial}{\partial r} u({\bf r}) - ik u({\bf r})\right] = 0.
\end{align}
In the following it is assumed that the incident field satisfies the homogeneous Helmholtz equation
\begin{align}
\Delta u_i +k^2  u_i = 0, \quad a\le |{\bf r}|\le b.
\end{align}

\subsection{Description of the algorithm}

For simplicity we describe the algorithm assuming that $a=0$ and $q$ is smooth, though the approach can be easily extended to the case where $a \neq 0$ and $q$ is piecewise smooth. As input the algorithm takes an outer radius $b,$ an accuracy $\epsilon,$ a frequency $k \in \mathbb{C},$ and functions which return the values of the incoming field $u_i$ and the potential $q.$ We note that $q$ is required to be radially-symmetric but $u_i$ is not. As output it returns the scattered field $u_s$ to within a specified accuracy $\epsilon.$ 

\begin{algorithm}[]~
\label{algo_qumu}

\noindent

\noindent Input: $b$, $k,$ $\epsilon,$ $q(r)$ and $u_i(x,y).$\\
\noindent Output: A function $u_s(x,y)$ valid for all $(x,y) \in \mathbb{R}^2.$

\begin{enumerate}
\item[Step 1.] Computing the Fourier coefficients of $u_i:$ Sample $u_i$ at 4000 equispaced points on the circle of radius $R.$ Denote these values by $\tilde{u}_j,$ $j=1,\dots,4000.$ Compute the FFT of $\tilde{u}_j$ and compute the smallest integer $M$ such that $|\tilde{u}_j|< \epsilon/10$ for all $j\in [M,4000-M].$ If no such $M$ exists double the number of sampling points.\\
\vspace{0.1 cm}

For each $m=0,\dots,M,$ calculate $\rho_m$ using the following procedure:
\item[Step 2.] Set $i=0,$ $ R_0 = b,$ $R_{\rm min} = \inf_{r>0} |J_m(rk)| < \epsilon/10.$ \\
While $R_i > R_{\rm min}$
\begin{enumerate}
\item Set $ r = \max\{R_i-\pi/k,R_{\rm min}\}.$
\item Set $J^* = \max_{[r,R_i]} |J_m(kr)|,$ $H^* = \max_{[r,R_i]} |H_m(kr)|,$ $q^* = \max_{[r,R_i]} |q(r)|.$ Compute the first 48 coefficients in the Chebyshev expansions of $q(r),$ $H_m(kr),$ $J_m(kr)$ on the interval $R_i-r.$ Denote them by $q^i_j,$ $J^i_j$ and $H^i_j$ respectively, where $j = 1,\dots,48.$ Calculate 
$$E := (R_i-r)\max_{j=13,\dots,48} \left\{ \frac{q^i_j}{q^*},\frac{J^i_j}{J^*},\frac{H^i_j}{H^*}\right\}.$$
If $E > \epsilon/10$ then set $r=\frac{1}{2}(R_i+r)$ and repeat.
\item For scattering from an external field set the $i$th source vector to be $f^i_j = J_m(kr_j)$ where $r_j,$ $j=1,\dots,48$ are the 48-point Chebyshev quadrature nodes translated and scaled to the interval $[r,R_i].$
\item Set $I_i=[r,R_i]$ and discretize the operator $F_I$ defined in equations (\ref{eqn:def_F}) using the quadrature nodes $r_1,\dots,r_{48}.$
\item Calculate $V_{I_i},$ $S_{I_i},$ $\alpha_\ell({I_i}),$ $\alpha_r({I_i}),$ $\chi_J(I_i)$ and $\chi_H(I_i)$ defined in (\ref{eqn:alp_def1}),(\ref{eqn:alp_def2}),(\ref{eqn:chi_def1}), and (\ref{eqn:chi_def2}) respectively.
\item Set $i=i+1$ and $R_i = r.$
\end{enumerate}
Let $N$ denote the total number of intervals required.
\item[Step 3.] Merging the intervals: starting from the outermost interval $[R_1,R_0]$ merge each interval using the results of Lemma \ref{lem:merge_lem}. In particular, for each $i=1,\dots,N$ compute and store the scattering matrices $S_{[R_i,R_0]},$ the expansion coefficients $\alpha_\ell([R_i,R_0]),$ $\alpha_r([R_i,R_0])$ and the source coefficients $\chi_J([R_i,R_0]),$ $\chi_H([R_i,R_0]).$
\item[Step 4.] Given $\chi_J([R_N,R_0]),$ and $\chi_H([R_N,R_0])$ starting from the innermost level compute  the incoming expansion coefficients $\phi_\ell([R_i,R_{i-1}]),$ $\phi_r([R_i,R_{i-1}])$ using the results of Lemma \ref{lem:merge_lem}. Finally, on each interval $[R_{i},R_{i-1}]$ obtain $\rho_m$ via equation (\ref{eqn:get_rho}).
\item[Step 5.] On each interval $I = [R_{i+1},R_i],$ obtain the solution $u_m$ via the equation
\begin{align}
u_m(r) =  \frac{f_m(kb)}{J_m(kb)}[F_I \rho_m(r) -i \frac{\pi}{2} \phi_\ell(I) H_m(kr)+ \phi_r(I) J_m(kr)].
\end{align}
\item[Step 6.] For $m <0$ obtain the solution $u_m$ are obtained in an analogous way.
\end{enumerate}
\end{algorithm}

\begin{remark1}
The complexity of the above algorithm is $ O(M \log M+ \bar{N}M)$ where $\bar{N}$ denotes the average number of discretization points per mode and grows like $b/k$ for large $k.$ For large $M$ a significant amount of time is spent solely in computing values of the Bessel and Hankel functions $J_m(kr)$ and $H_m(kr),$ respectively. This computational cost can be dramatically reduced by using asymptotic formulae for $J_m(kr)$ and $H_m(kr)$ (see, for example, \cite{abram}) as well as non-oscillatory phase functions \cite{brem_nonosc}.
\end{remark1}

\section{Numerical results}\label{sec:nume}
The algorithm described in Algorithm \ref{algo_qumu} which solves the two-dimensional scattering problem from radially-symmetric potentials was implemented in GFortran and experiments were run on a  2.7 GHz Apple laptop with 8 Gb RAM. Specifically, we consider scattering from a disk of radius $2\pi$ with the following potentials: 
\begin{enumerate}
\item A Gaussian bump with potential $q(r) = e^{-r^2}$ (see Figure \ref{fig:gaupot})
\item A random discontinuous media: 20 points are uniformly sampled from the interval $[0,2\pi].$ At each such point the potential switches from $0$ to $1$ or vice versa (see Figure \ref{fig:dispot}) 
\item An Eaton lens (\cite{eaton}) with potential $q$ satisfying the equation
$$ 1+q(r)= \frac{2\pi}{\sqrt{1+q(r)} r}+  \sqrt{\left(\frac{2\pi}{\sqrt{1+q(r)} r}\right)^2-1},$$
(see Figure \ref{fig:eatpot}) 
\item A Luneburg lens (\cite{luneburg}) with potential $q(r) = 1-\frac{r^2}{4\pi^2},$ (see Figure \ref{fig:lunpot}).
\end{enumerate}
The incoming field is chosen to be one of the following two functions:
\begin{enumerate}
\item An incoming plane wave: 
$$f({\bf r}) = \exp (ikx/2+i\sqrt{3} ky/2),$$
where ${\bf r} = (x,y).$
\item A Gaussian beam: 
$$f({\bf r}) = H_0 \left(k \sqrt{(x+16-8i)^2+y^2} \right)\,\,e^{-7.859 k}, $$
where ${\bf r} = (x,y).$
\end{enumerate}
In all experiments the accuracy $\epsilon$ was set to $10^{-13}.$ The number of modes needed and the total time per solve are summarized in Table \ref{tabl_erro}. Plots of the magnitude of the field $|u|$ are given in Figure \ref{fig:gausol}, Figure \ref{fig:dissol}, and Figure \ref{fig:eatsol}. In order to demonstrate the accuracy of the solution we compute the error function $E: \mathbb{R}^2 \to [0,\infty)$ defined by
\begin{align}
E({\bf x}) = \left|\frac{1}{k^2} \Delta u({\bf r}) + (1+Q({\bf r})) u({\bf r}) \right|,
\end{align}
log plots of which are shown in Figure \ref{fig:gauerr}, Figure \ref{fig:diserr}, and Figure \ref{fig:eaterr}. Finally, Figure \ref{fig:modedep} shows the typical behavior of the number of spatial discretization nodes used as a function of the mode number $m.$

\begin{singlespace}
\begin{table}
\begin{tabular}{c|c|c|c|c}
\makecell{Potential $q$} &Source & Frequency $k$ & \makecell{Number  of modes} & \makecell{Solve  time (s)}  \\
\hline
Gaussian& plane wave & 100 & 711 & 36.07\\
Random& plane wave & 30 & 245 & 7.551\\
Eaton& Gaussian beam & 30 & 233 & 6.684

\end{tabular}
\caption{Numerical results for time-harmonic scattering from radially-symmetric potentials.}\label{tabl_erro}
\end{table}
\end{singlespace}

        \begin{figure}
        \centering
        \includegraphics[width=0.4\textwidth]{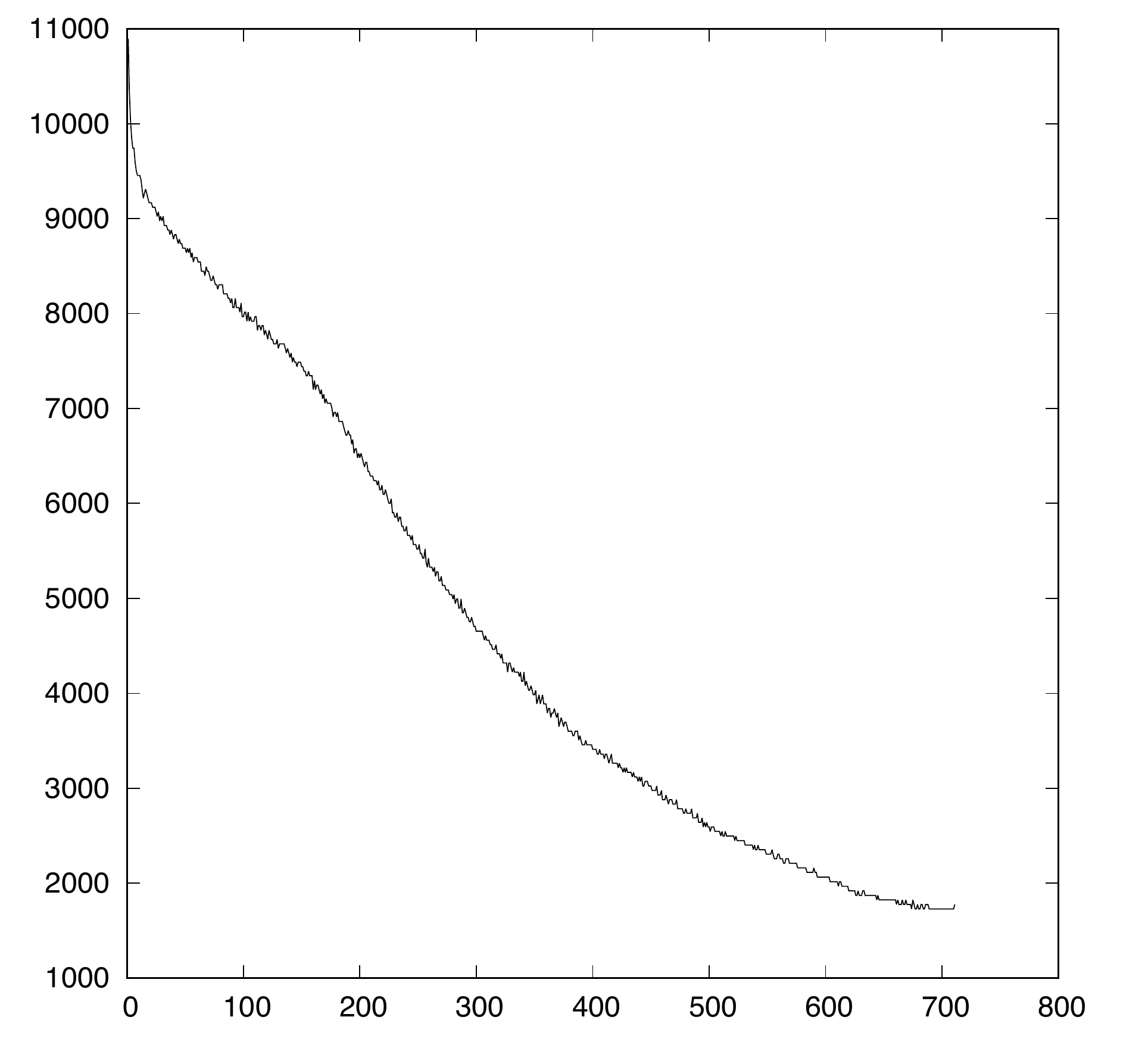}
        \caption{Number of radial points versus mode number}
        \label{fig:modedep}
\end{figure}

\begin{figure}
    \centering
    \begin{subfigure}[b]{0.45\textwidth}
        \includegraphics[width=0.85\textwidth]{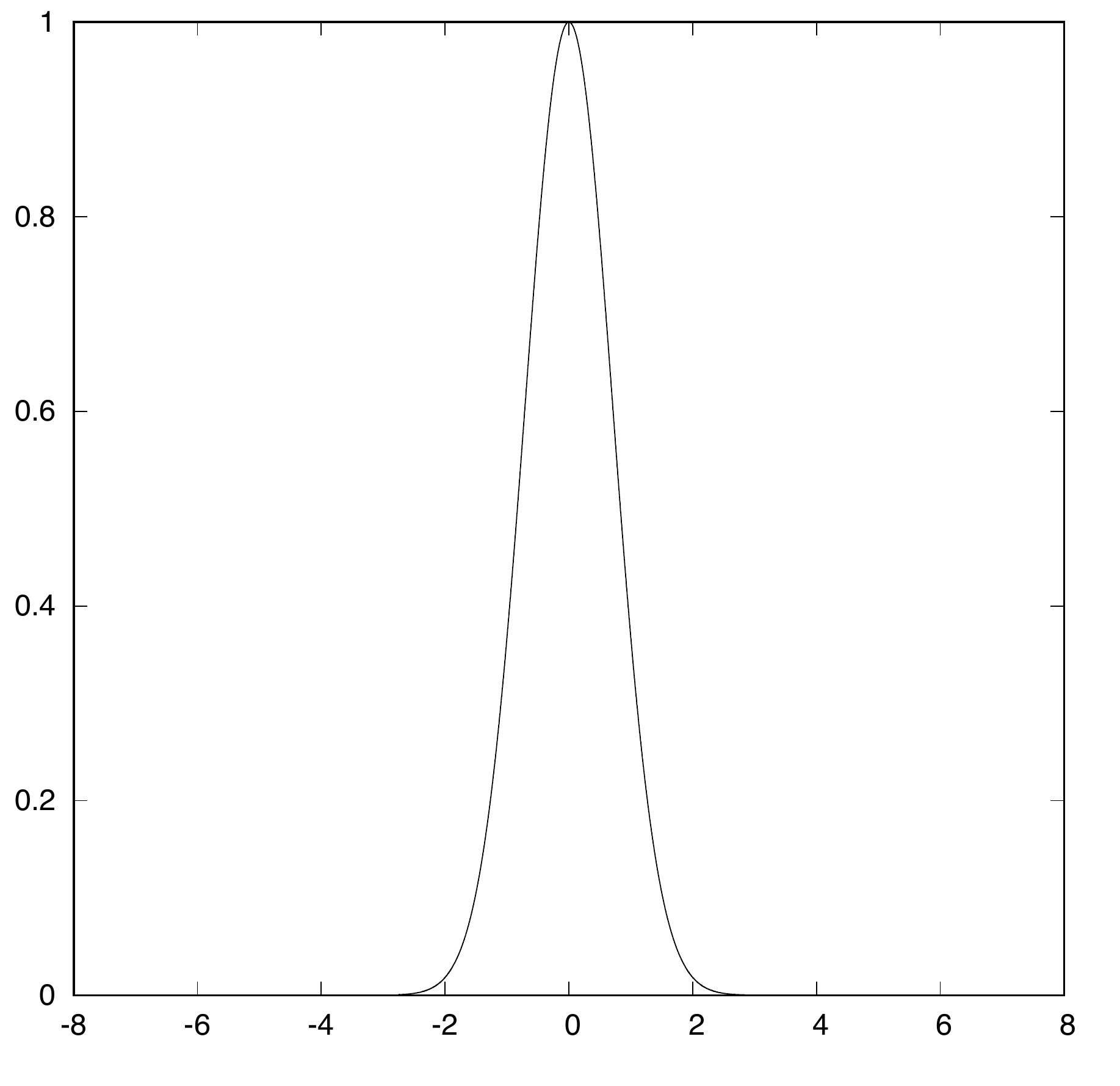}
        \caption{Radial section of the Gaussian bump}
        \label{fig:gaupot}
    \end{subfigure}
    ~ 
    \begin{subfigure}[b]{0.45\textwidth}
        \includegraphics[width=\textwidth]{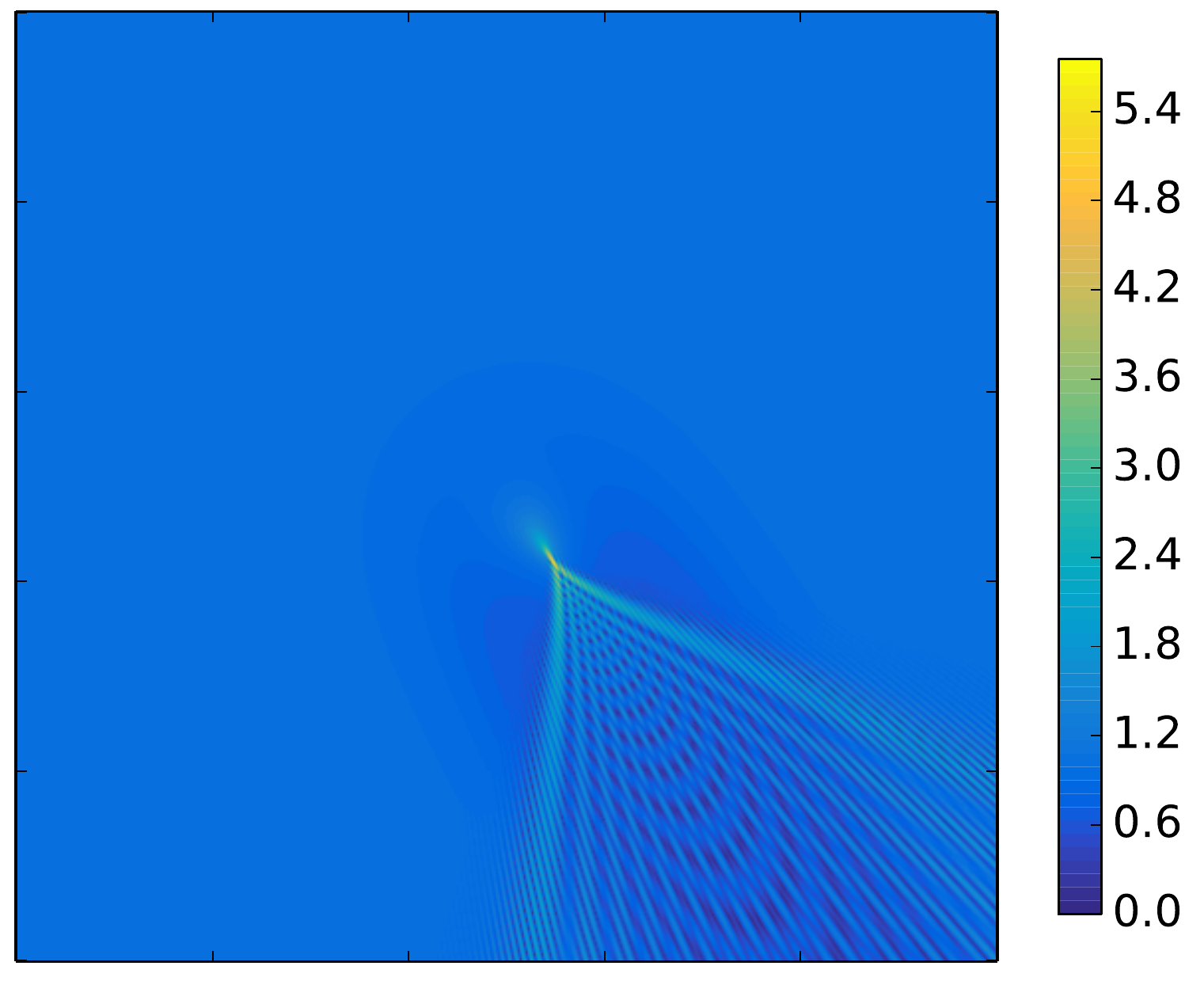}
        \caption{Magnitude of the total field $|u|$}
        \label{fig:gausol}
    \end{subfigure}
    \vspace{0.5 cm}
    
    ~ 
    \begin{subfigure}[b]{0.45\textwidth}
        \includegraphics[width=\textwidth]{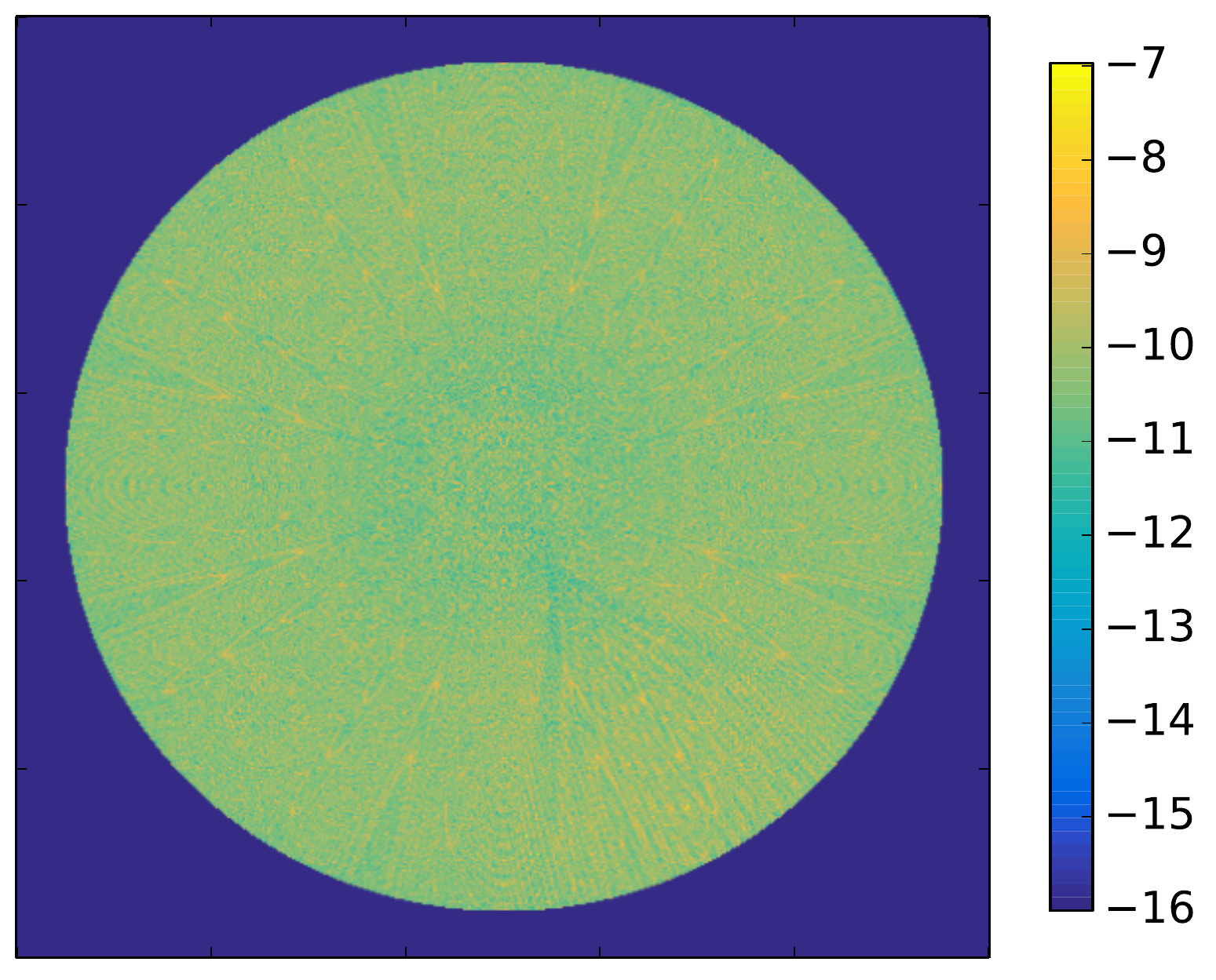}
        \caption{Error in the solution ($\log_{10}$): \\$\hspace{ 0.5 cm}\log_{10}( |\Delta u +k^2(1+q(\|{\bf r}\|))u|/k^2)$}
        \label{fig:gauerr}
    \end{subfigure}
    \caption{Numerical results for the scattering of a plane wave from a Gaussian potential with a standard deviation of $1$ at frequency $k =100.$}\label{fig:gauss}
\end{figure}

\begin{figure}
    \centering
    \begin{subfigure}[b]{0.45\textwidth}
    \centering
        \includegraphics[width=0.85 \textwidth]{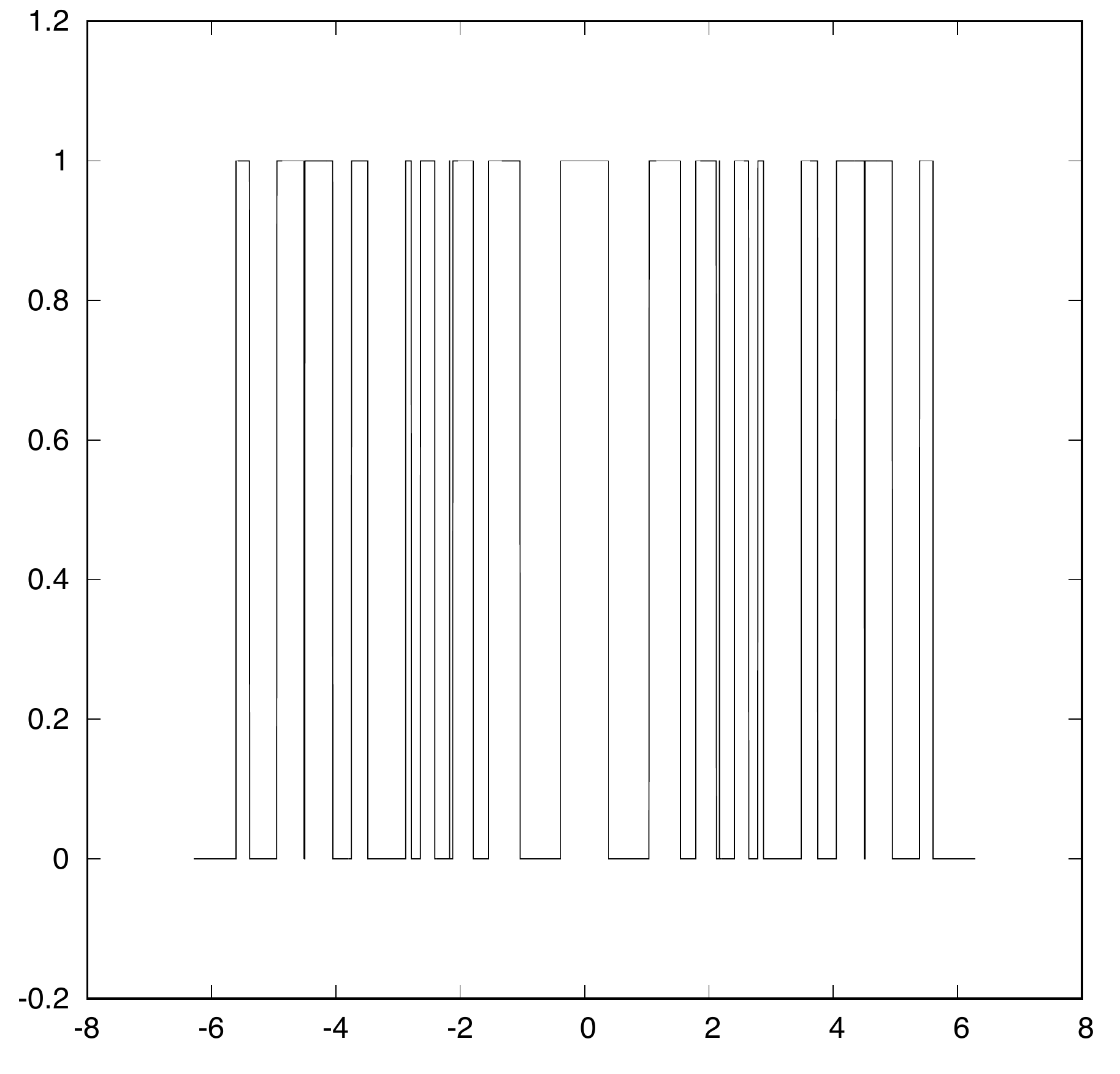}
        \caption{Radial section of the random discontinuous medium.}
        \label{fig:dispot}
    \end{subfigure}
    ~ 
    \begin{subfigure}[b]{0.45\textwidth}
        \includegraphics[width=\textwidth]{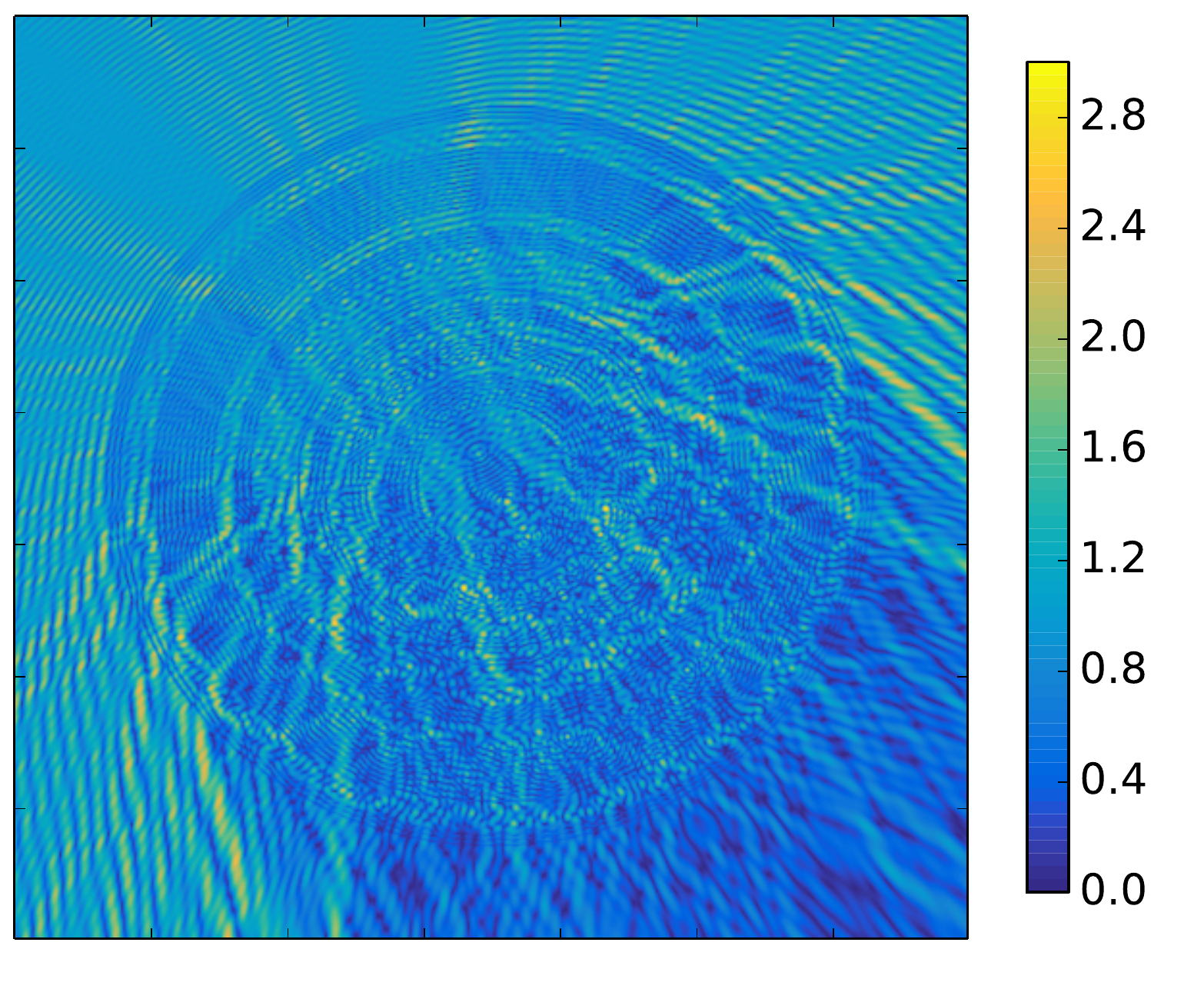}
        \caption{Magnitude of the total field $|u(\bf r)|$}
        \label{fig:dissol}
    \end{subfigure}
    ~ 
    \begin{subfigure}[b]{0.45\textwidth}
    \centering
        \includegraphics[width=\textwidth]{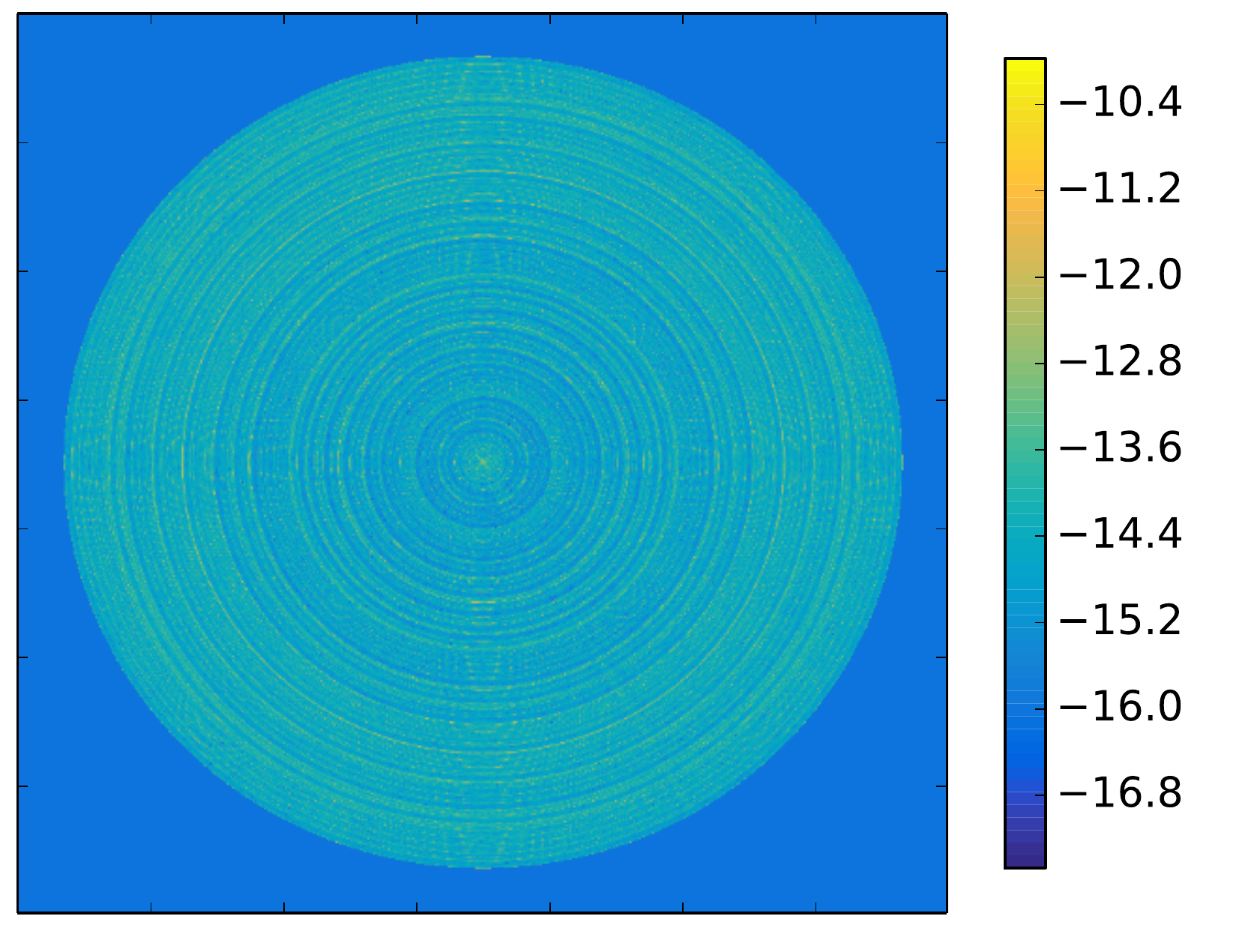}
        \caption{Error in the solution ($\log_{10}$): \\$\hspace{ 0.5 cm}\log_{10} (|\Delta u +k^2(1+q({\bf r}))u|/k^2)$ generated in extended precision}
        \label{fig:diserr}
    \end{subfigure}
    \caption{Numerical results for the scattering of a plane wave from a discontinuous potential with diameter $4\pi$ at a frequency of $k=30.$}\label{fig:disco}
\end{figure}

\begin{figure}
    \centering
    \begin{subfigure}[b]{0.45\textwidth}
    \centering
        \includegraphics[width=0.85 \textwidth]{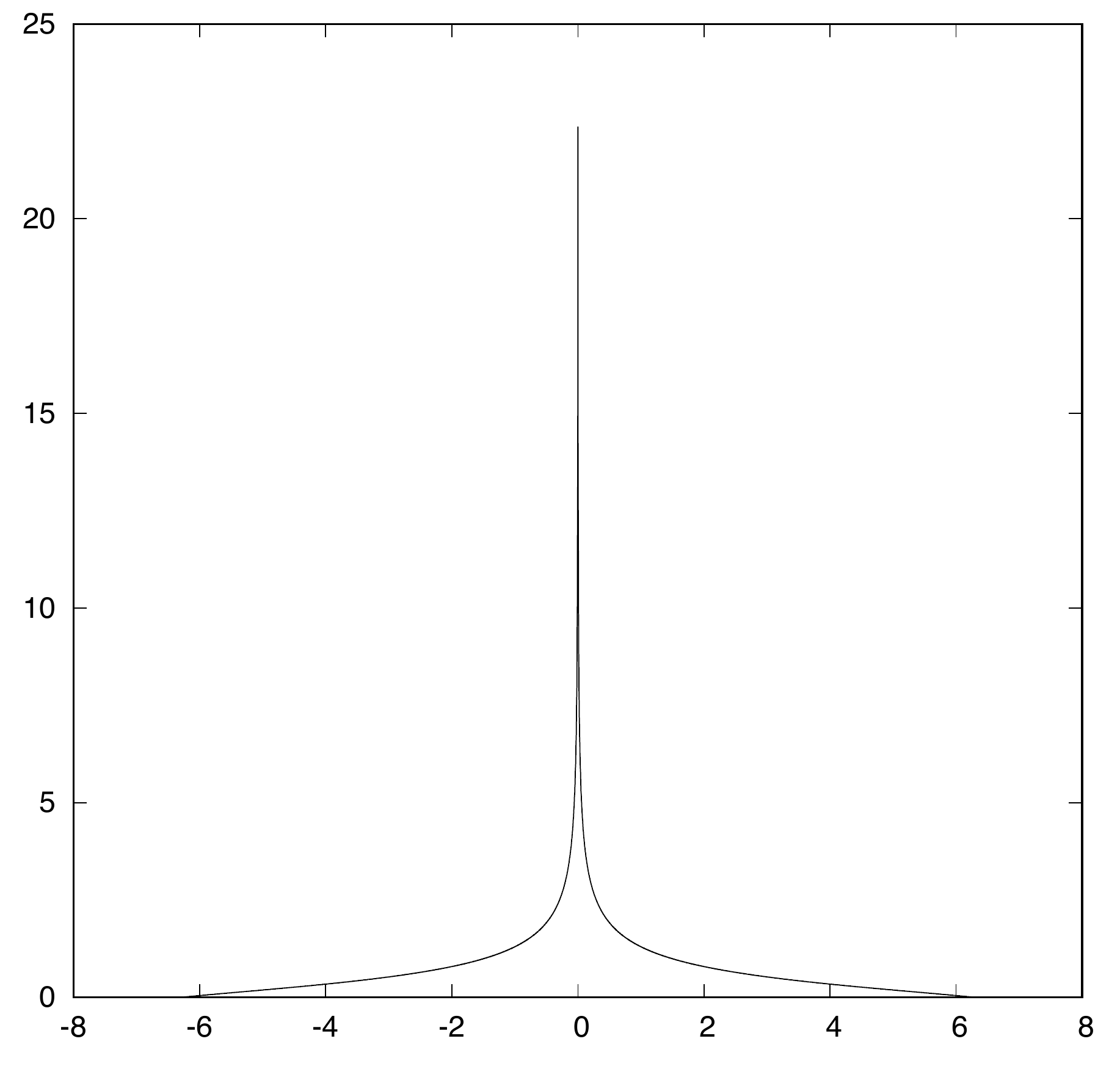}
        \caption{Radial section of the Eaton lens}
        \label{fig:eatpot}
    \end{subfigure}
    ~ 
    \begin{subfigure}[b]{0.45\textwidth}
        \includegraphics[width=\textwidth]{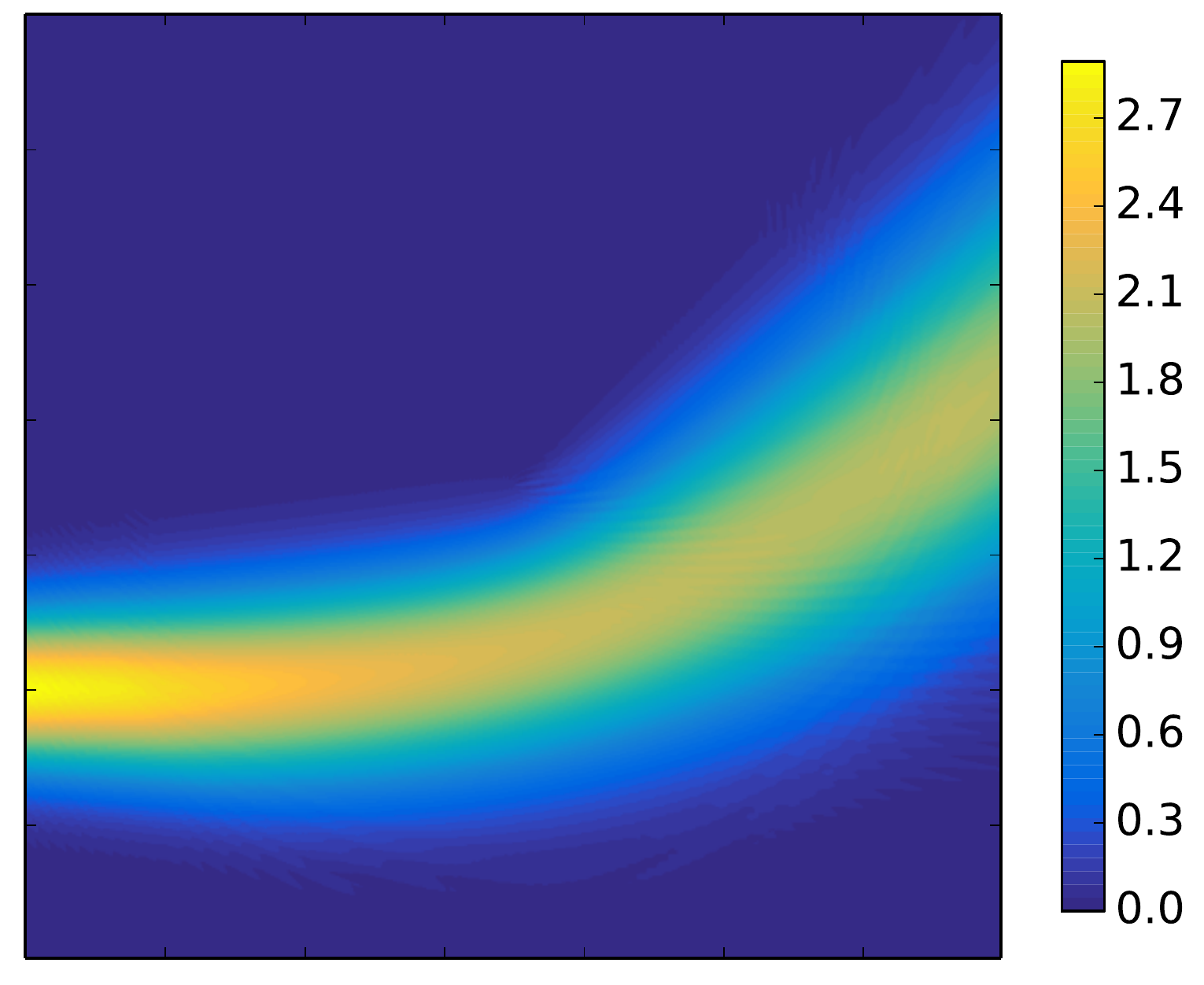}
        \caption{Magnitude of the total field $|u|$}
        \label{fig:eatsol}
    \end{subfigure}
    ~ 
    \begin{subfigure}[b]{0.45\textwidth}
    \centering
        \includegraphics[width=\textwidth]{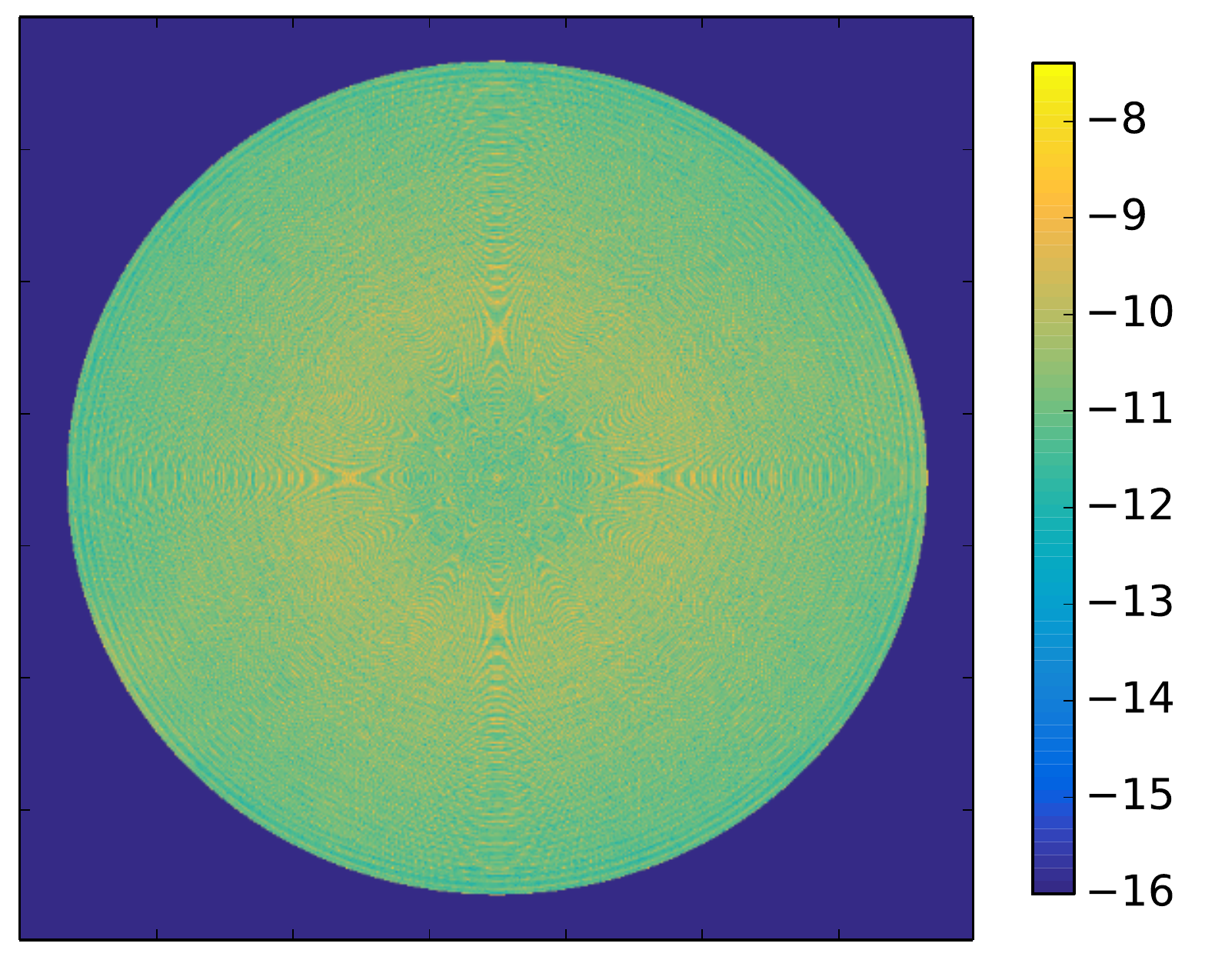}
        \caption{Error in the solution ($\log_{10}$): \\$\hspace{ 0.5 cm}\log_{10} (|\Delta u +k^2(1+q({\bf r}))u|/k^2)$}
        \label{fig:eaterr}
    \end{subfigure}
    \caption{Numerical results for the scattering of a Gaussian beam from an Eaton lens with a diameter of $4 \pi$ at a frequency of $k=30.$}\label{fig:gaubeam}
\end{figure}

\subsection{Time domain problems}
In this section we present numerical illustrations of the application of Algorithm \ref{algo_qumu} to two-dimensional time-dependent scattering problems. For time-dependent scattering problems the displacement $u:\mathbb{R}^2 \times [0,\infty) \to \mathbb{C}$ satisfies the initial value problem
\begin{align}
&\Delta u({\bf r},t) - (1+Q({\bf r})) \frac{\partial^2}{\partial t^2} u({\bf r},t) = f({\bf r},t)\\
&u({\bf r},0) = 0\\
&\frac{\partial}{\partial t} u({\bf r},0) = 0.
\end{align}
In the following we assume that the potential $Q:\mathbb{R}^2 \to [q_0,q_1]$ is a radially-symmetric function compactly supported on a ball of radius $b.$ Moreover, we assume that the source $f$ has the following two properties:
\begin{enumerate}
\item $f({\bf r},t) =0$ for all $\|{\bf r}\| \le b,t\in[0,\infty),$
\item for all ${\bf r} \in \mathbb{R}^2,$ $f({\bf r},t)$ is a $C^\infty$ function of $t,$ compactly supported in some interval $[0,T].$
\end{enumerate}

If $\tilde{u}({\bf r},k)$ denotes the Fourier transform of $u$ with respect to time evaluated at frequency $k,$
\begin{align}
\tilde{u}({\bf r},k) = \frac{1}{(2\pi)^2}\int_{-\infty}^\infty e^{-ikt}\,u({\bf r},t)\,{\rm d}t,
\end{align}
then $\tilde{u}$ satisfies the Helmholtz equation (see, for example, \cite{born})
\begin{align}
\Delta \tilde{u}({\bf r},k) + k^2 (1+Q({\bf r})) \tilde{u}({\bf r},k)=\tilde{f}({\bf r},k)
\end{align}
where $\tilde{f}$ is the Fourier transform of the source $f.$ Moreover, given $\tilde{u}$ one can compute $u$ via the inverse Fourier transform 
\begin{align}\label{eqn:tdep_int}
{u}({\bf r},t) = \int_{-\infty}^\infty e^{ikt}\,\tilde{u}({\bf r},k)\,{\rm d}k.
\end{align}
We observe that since $f$ is smooth in time its Fourier transform decays rapidly with $k.$ In particular, for some constant $K$ depending on $f$ the interval of integration $(-\infty,\infty)$ in (\ref{eqn:tdep_int}) can be replaced by a finite interval $[-K,K]$ with essentially no loss in accuracy.

As an example to illustrate this approach we consider the problem of scattering from a Luneburg lens (see Figure \ref{fig:lunpot}) with an incoming source $f$ given by
$$f({\bf r},t) =\sqrt{8}\, \delta(x-10) \delta(y-10) e^{-4(t-10)^2}$$
where ${\bf r} = (x,y).$ The solution was calculated by solving the problem at $516$ frequencies in the interval $[-16,16].$ The intervals $[2,16]$ and $[-16,-2]$ were discretized using a 200-point Gauss-Legendre quadrature while the interval $[-2,2]$ was discretized using a custom generalized Gaussian quadrature.

\begin{figure}
    \centering
    \begin{subfigure}[b]{0.45\textwidth}
    \centering
        \includegraphics[width=0.85 \textwidth]{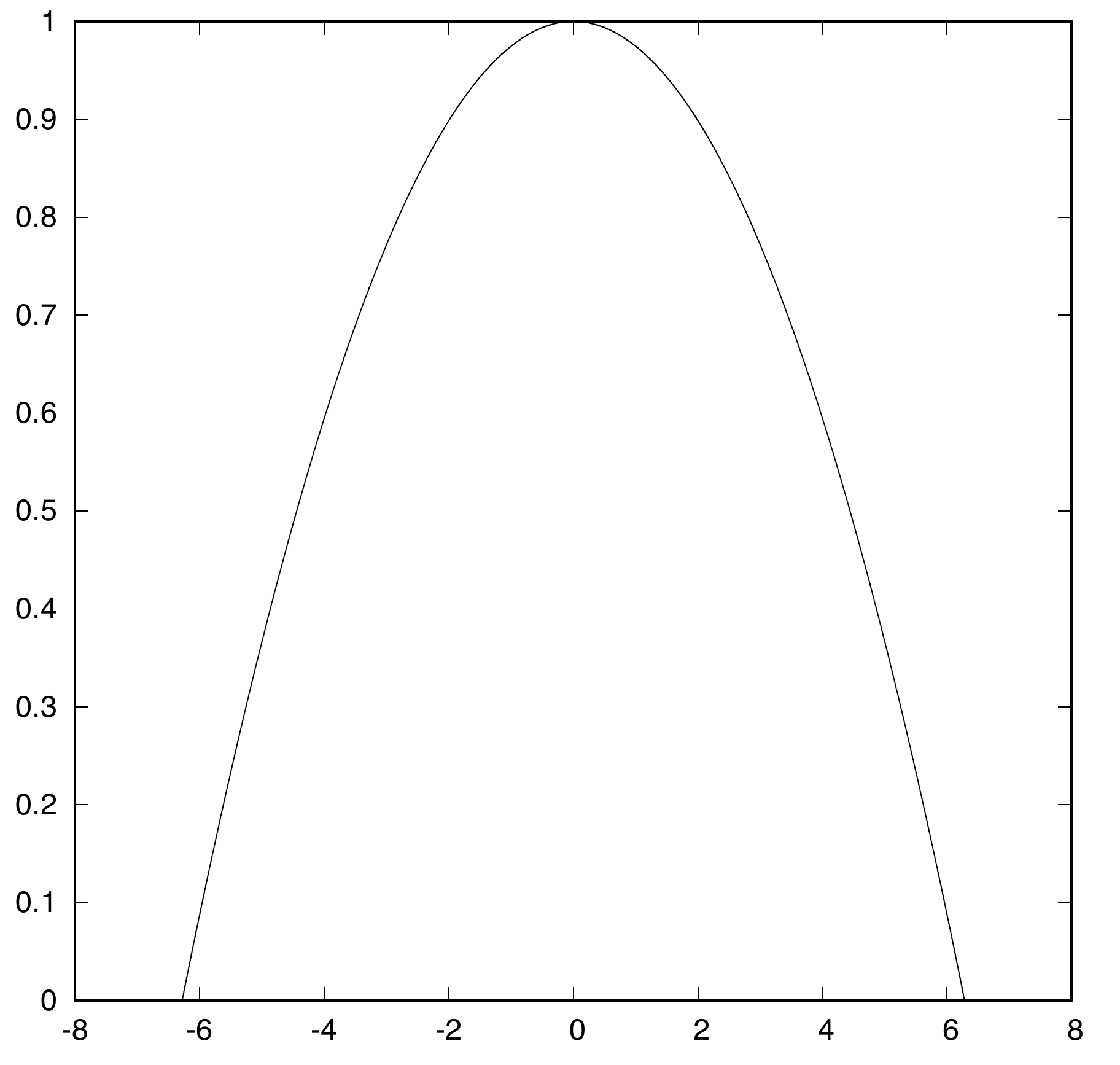}
        \caption{Radial section of the Luneburg lens}
        \label{fig:lunpot}
    \end{subfigure}
    \centering
    \begin{subfigure}[b]{0.45\textwidth}
    \centering
        \includegraphics[width=1 \textwidth]{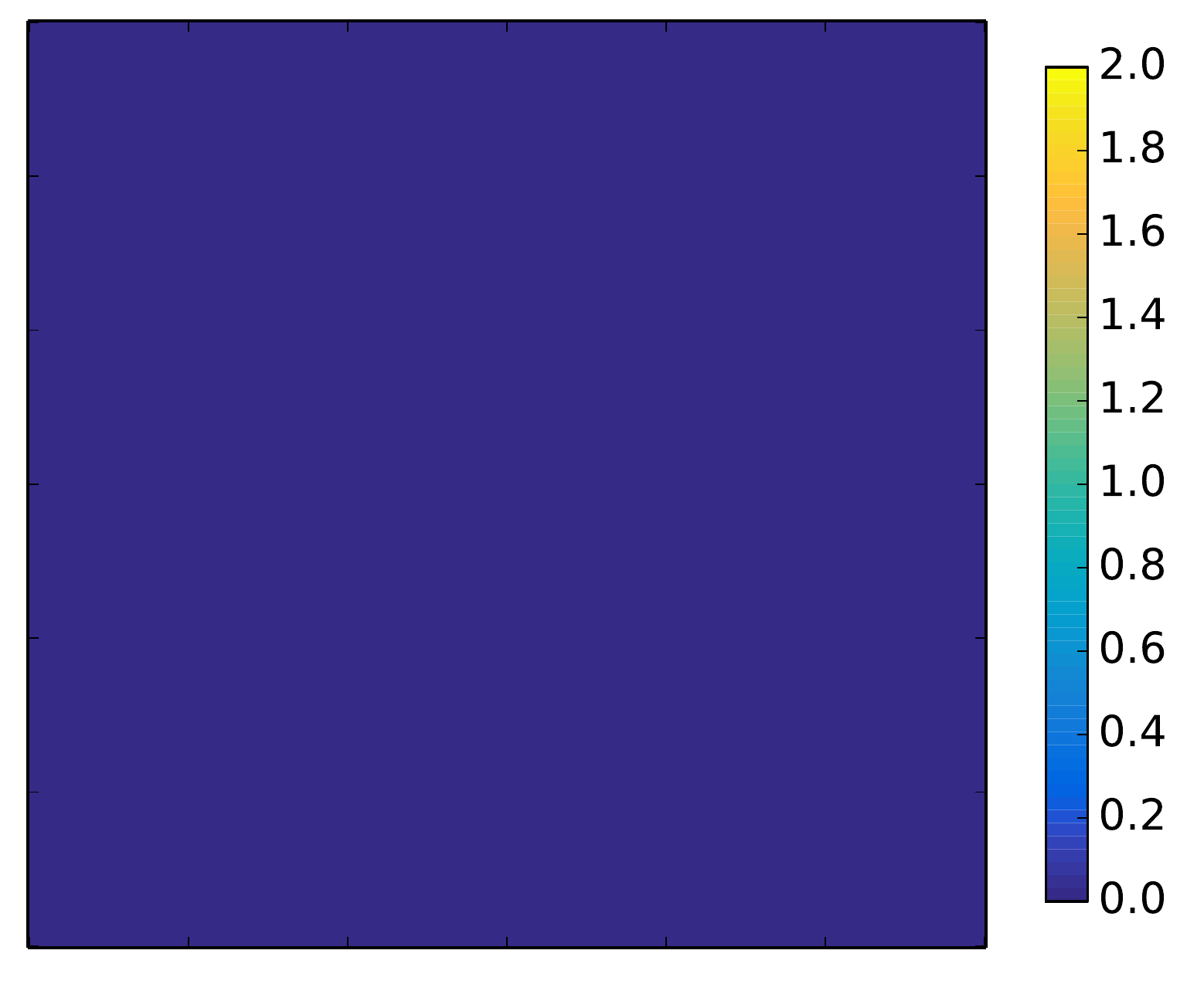}
        \caption{Magnitude of the field at $t=13.6$ s}
        \label{fig:timp1}
    \end{subfigure}
    ~ 
    \begin{subfigure}[b]{0.45\textwidth}
        \includegraphics[width=\textwidth]{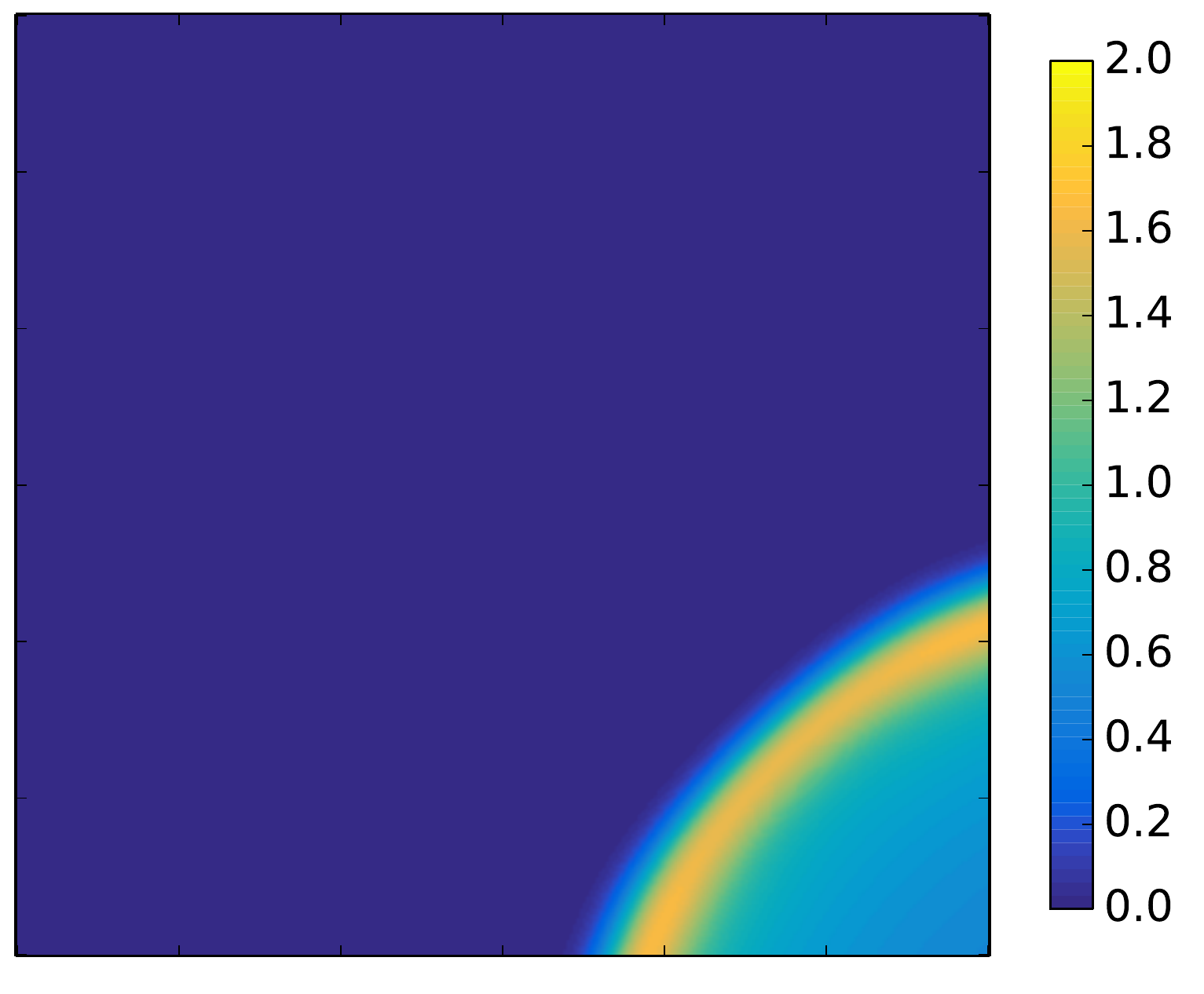}
        \caption{Magnitude of the field at $t = 19$ s}
        \label{fig:timp2}
    \end{subfigure}
    ~ 
    \begin{subfigure}[b]{0.45\textwidth}
    \centering
        \includegraphics[width=\textwidth]{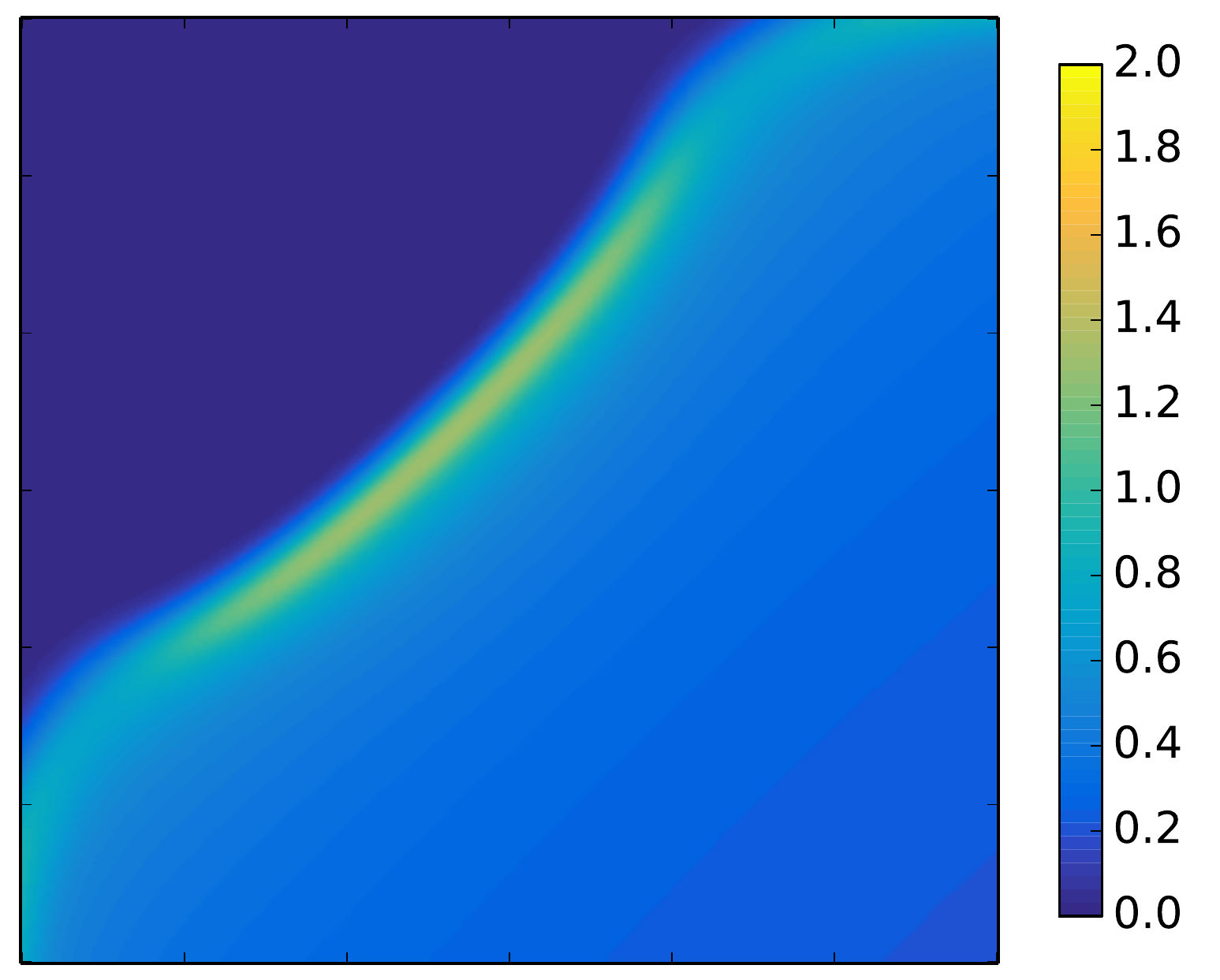}
        \caption{Magnitude of the field at $t=28$ s}
        \label{fig:timp3}
    \end{subfigure}
        \begin{subfigure}[b]{0.45\textwidth}
    \centering
        \includegraphics[width=\textwidth]{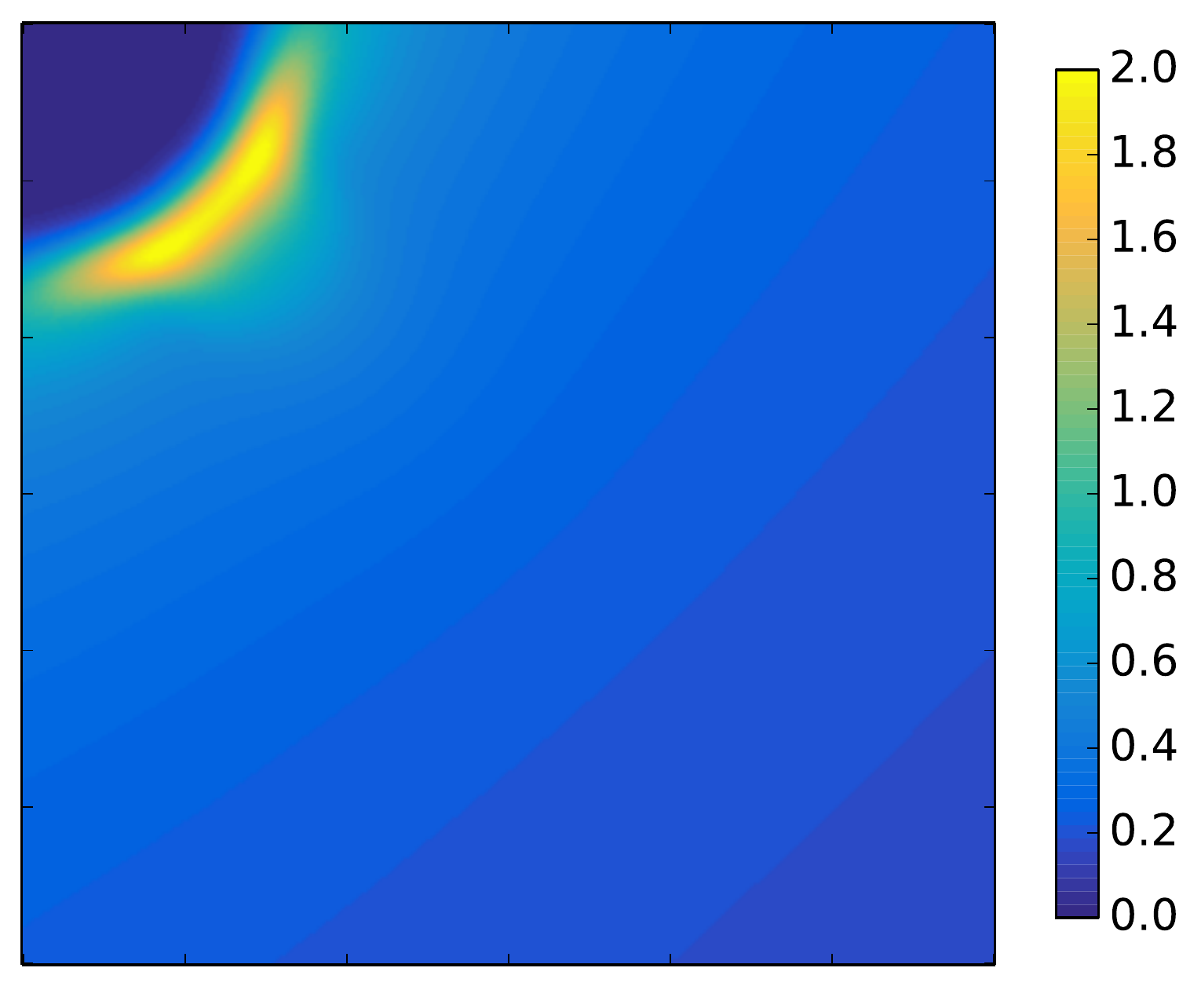}
        \caption{Magnitude of the field at $t = 34$ s}
        \label{fig:timp4}
    \end{subfigure}
        \begin{subfigure}[b]{0.45\textwidth}
    \centering
        \includegraphics[width=\textwidth]{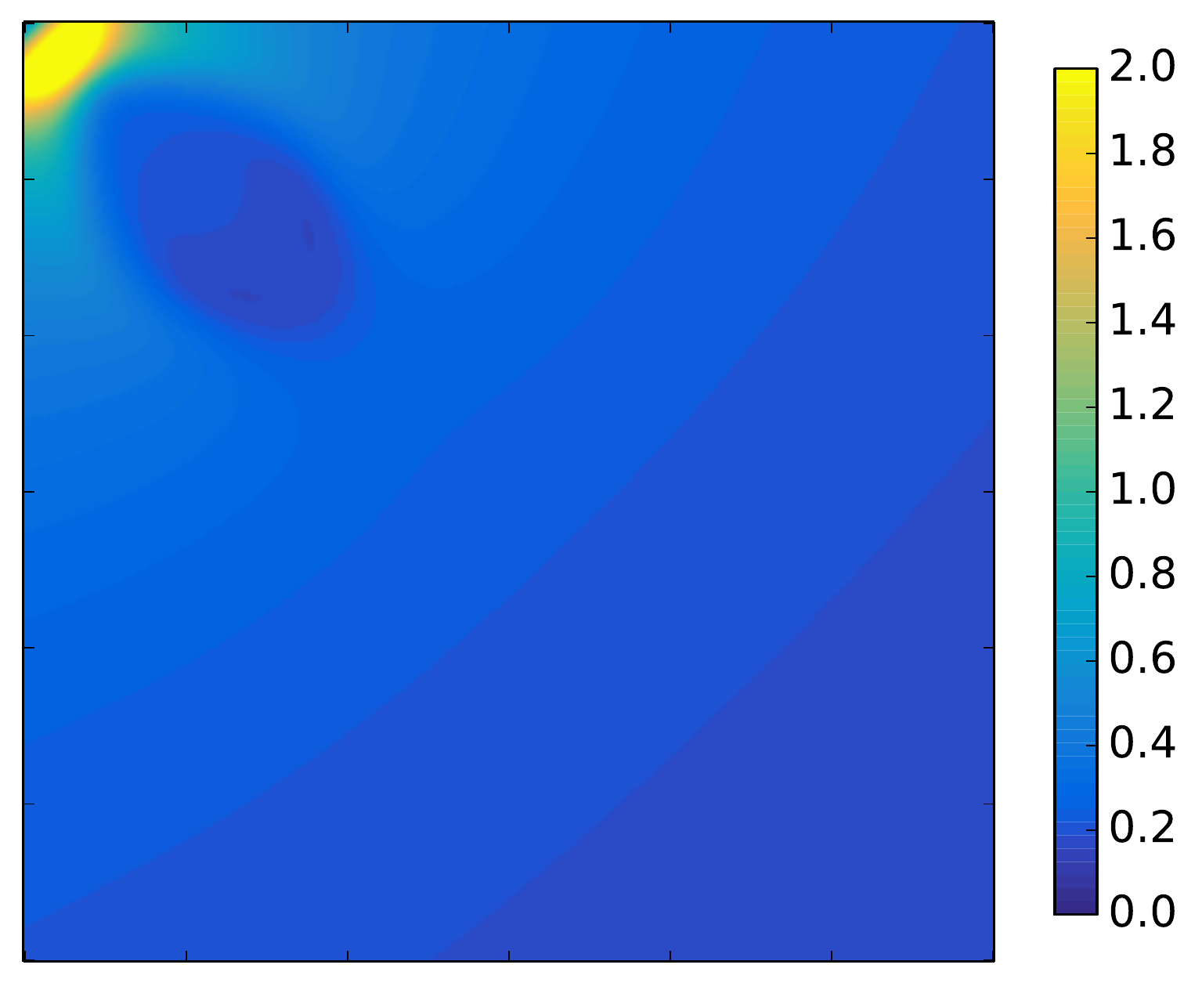}
        \caption{Magnitude of the field at $t = 37$ s }
        \label{fig:timp5}
    \end{subfigure}
    \caption{Numerical results for the scattering of a wave from a Luneburg lens.}\label{fig:timedep}
\end{figure}

\section{Conclusion and future work}\label{sec:concl}
In this paper we described a fast, adapative, simple, and accurate method for computing the scattering from a radially-symmetric body in two dimensions. The algorithm is based on taking Fourier series in the angular variable and solving the resulting equations mode by mode using a fast adaptive solver based on scattering matrices. Numerical experiments were performed which demonstrate the performance of the algorithm. We observe that a similar approach can be employed for three-dimensional radially-symmetric scattering problems as well as for waveguides with constant cross-sectional parameters. Finally, a natural extension to this algorithm would be to collections of compactly-supported radially-symmetric scatterers, which arise in problems in optics and the study of wave propagation in disordered media.
\vspace{2 cm}

Both authors were supported in part by AFOSR FA9550-16-1-0175 and by the ONR N00014-14-1-0797.

\bibliographystyle{siam}

\end{document}